\documentclass[12pt]{amsart}
\usepackage{amssymb,latexsym, amsthm,epsf}

\newtheorem{thm}{Theorem}[section] 

\newtheorem{cor}[thm]{Corollary}

\newtheorem{exa}[thm]{Example}
\newtheorem{prop}[thm]{Proposition}
\newtheorem{rem}[thm]{Remark}
\newtheorem{algo}[thm]{Algorithm}

\def\R{{\mathbb R}}
\def\P{{\mathbb P}}

\def\({\left(}
\def\){\right)}

\long\def\forget#1\forgotten{}

\newif \iffurther 
\furtherfalse

\newif \iffurther 
\furtherfalse

\newif\ifXY 
\XYfalse    

\ifXY
\usepackage{xy}
\fi
\ifXY
\xyoption{all}
\fi

\begin{document}

\title{Eulerian partitions for configurations of skew lines}

\author[Roland Bacher and David Garber]{Roland
  Bacher$^{1}$ and David Garber$^2$}

\stepcounter{footnote}
\footnotetext{Corresponding author.}
\stepcounter{footnote}
\footnotetext{Partially supported by the Chateaubriand postdoctoral
  fellowship funded by the French government.}

\address{Roland Bacher, Institut Fourier, BP 74, 38402 Saint-Martin D'Heres CEDEX,
  France.}\email{roland.bacher@ujf-grenoble.fr}
\address{David Garber, Departement of Applied Mathematics, Faculty of
  Sciences, Holon Institute of Technology, 52 Golomb st., PO
  Box 305, 58102 Holon, Israel.}
\email{garber@hit.ac.il}

\keywords{Configurations of skew lines, linking matrix,
Switching graph, Eulerian graph,
Spindles, permutation}

\begin{abstract}
In this paper, which is a complement of \cite{BG}, we study a few
elementary invariants for configurations of skew lines, as 
introduced and analyzed first by Viro and
his collaborators. We slightly simplify the exposition of some 
known invariants and use them to define a natural
partition of the lines in a skew configuration.

We also describe an algorithm which constructs a spindle-permutation for
a given switching class, or proves non-existence of such a spindle-permutation.
\end{abstract}

\maketitle





\section{Introduction}

{\it A configuration of $n$ skew lines in $\R^3$} or a {\it skew
configuration} or an {\it interlacing of skew lines} is a set
of $n$ non-intersecting lines in $\R^3$
containing no pair of parallel lines.

Skew configurations are only considered up to
{\it rigid isotopies}
(continuous deformations of
skew configurations or, equivalently, isotopies of the ambient space
under which lines remain pairwise skew lines).

The study and classification of configurations of skew lines
up to isotopy was initiated by Viro \cite{V} and pursued
by Viro, Mazurovski{\u\i},
Borobia-Mazurovski{\u\i}, Drobotukhina and Khashin, see
for example \cite{BM1,BM2,D,K,MaR,Ma,VD}.
The survey paper \cite{VD} (and its
updates available on the authors web-sites and on the
arXiv) contains historical information and is
a good introduction into the subject and its higher-dimensional
generalizations.
Most of these results are also exposed in the survey paper
\cite{CP} from which we borrow some terminology not used in 
the original work of Viro's school.

\medskip
A {\it spindle} (or {\it isotopy join} or {\it horizontal configuration})
is a particularly nice configuration of skew lines
in which all lines intersect an oriented additional
line $A$, called the {\it axis} of the spindle.
Its isotopy class is completely described by a
{\it spindle-permutation} $\sigma:\{1,\dots,n\}\longrightarrow \{1,\dots,n\}$
encoding the order in which an open half-plane
revolving around its boundary $A$ intersects the lines
during a half-turn (see Section
\ref{spindles} for the precise definition).
A {\it spindle-configuration} is a skew configuration isotopic to a spindle.

Three types of combinatorial moves (described in Section
\ref{spindles}) of a spindle permutation yield
isotopic spindles-configurations and generate an 
equivalence relation, called the 
{\it spindle-equivalence} relation, on permutations
of $\{1,\dots,n\}$.

Isotopy classes of spindle-configurations are well understood 
by the following combinatorial description, given in \cite{BG}:

\begin{thm} \label{main}
Two spindle-permutations $\sigma,\sigma'$
give rise to isotopic spindle-configurations
if and only if $\sigma$ and $\sigma'$ are spindle-equivalent.
\end{thm}

Orienting and labeling all lines of a skew configuration, one gets a {\it
linking matrix} encoding isotopy classes for pairs of oriented skew lines.
The associated {\it switching class} or {\it homology equivalence class}
is independent of labels and orientations.
A result of Khashin and Mazurovski{\u\i} \cite[Theorem 3.2]{KM}
states that homology-equivalent spindles (spindles defining
the same switching class) are isotopic. Thus we have:

\begin{cor} \label{maincor}
Two spindle-permutations define the same switching class if
and only if they are spindle-equivalent.
\end{cor}

In this paper we define the Euler partition for a 
switching class.
The definition depends on the parity of the order and can be 
refined to an Euler tree for switching classes of even order.

The last part of the paper describes an algorithm for computing a
spindle-permutation (or proving its non-existence) for a given 
switching class.

\medskip
The sequel of this paper is organized as follows:
Sections \ref{mat_invar} and \ref{linking} are devoted to (various aspects of)
switching classes.
Section \ref{euler} deals with the Eulerian partition induced by the switching classes. Section \ref{spindles} introduces spindle-structures.
In Section \ref{algo} we describe an algorithm which computes (or proves
non-existence of) a spindle-permutation (which is unique up to
spindle-equi\-valence
by Corollary \ref{maincor}) having a linking matrix of
given switching class.


\section{Linking matrices and switching classes}\label{mat_invar}

Pairs of oriented under- or
over-crossing curves (as arising for instance from oriented
knots and links) can be encoded by signs $\pm 1$ as shown in 
Figure \ref{fig1}.

\begin{figure}[h]
\epsfysize=3cm
\centerline{\epsfbox{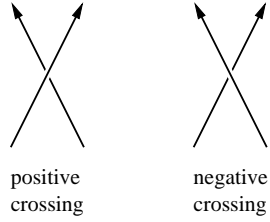}}
\caption{Positive and negative crossings}
\label{fig1}
\end{figure}

The {\it sign} or {\it linking number}
$\hbox{lk}(L_A,L_B)\in\{\pm 1\}$ between two oriented skew lines
$L_A,L_B\subset \R^3$ was introduced by Viro \cite{V}.
The {\it linking matrix} of a configuration involving $n$
oriented and labeled skew lines
$L_1,\dots,L_n$ is the symmetric $n\times n$
matrix $X$ with diagonal coefficients $x_{i,i}=0$ and
$x_{i,j}=\hbox{lk}(L_i,L_j)\in\{\pm 1\}$ for $i\not= j$.

\begin{figure}[h]
\epsfysize=5cm
\centerline{\epsfbox{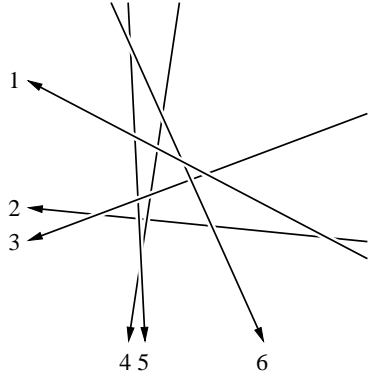}}
\caption{A configuration of $6$ labeled and oriented skew lines}
\label{fig2}\end{figure}

Figure \ref{fig2} shows a labeled and oriented configuration of six skew
lines with linking matrix
$$X=\left(
\begin{array}{rrrrrr}
0 &  1 &  1 &  1 &  1 &  1 \\
1 &  0 & -1 & -1 & -1 & -1 \\
1 & -1 &  0 &  1 &  1 & -1 \\
1 & -1 &  1 &  0 & -1 & -1 \\
1 & -1 &  1 & -1 &  0 & -1 \\
1 & -1 & -1 & -1 & -1 &  0
\end{array}
\right)\ .
$$

Two symmetric matrices $X$ and $Y$ are {\it switching-equivalent} if
$$Y=D\ P^t\ X\ P\ D$$
where $P$ is a permutation matrix and $D$ is a diagonal matrix with
$d_{i,i}\in\{\pm 1\}$. Since $P$ and $D$
are orthogonal, we have
$(PD)^{-1}=D^tP^t=DP^t$. Switching-equivalent matrices are thus conjugate and
have the same characteristic polynomial.

\begin{prop} All linking matrices of a fixed configuration of skew lines
are switching-equivalent.
\end{prop}

\begin{proof}
Relabeling the lines conjugates a linking matrix $X$
by a permutation matrix. Inverting the orientation
of some lines amounts to conjugation by a diagonal $\pm 1$ matrix.
\end{proof}

\begin{rem}
The terminology \lq\lq
switching classes'' (many authors speak of \lq\lq two-graphs''
which is the standard terminology for the underlying combinatorial object)
is motivated by the following combinatorial interpretation and
definition of switching classes.

Two finite simple (loopless and no multiple edges)
graphs $\Gamma_1=(V,E_1)$
and $\Gamma_2=(V,E_2)$ 
are {\bf switching-related} with respect to a subset of
vertices $V_-\subset V$ if their edge-sets
$E_1,E_2$ coincide on $(V_-\times V_-)\cup \big((V\setminus V_-)\times
(V\setminus V_-)\big)$ and are complementary on $\big(V_-\times (V\setminus V_-)\big)\cup \big((V\setminus V_-)
\times V_-\big)$. A {\it switching class} of graphs
is an equivalence class of switching-related graphs, see Figure \ref{switching}
for two graphs in a common switching class.

\begin{figure}[h]
\epsfysize=2.5cm
\epsfbox{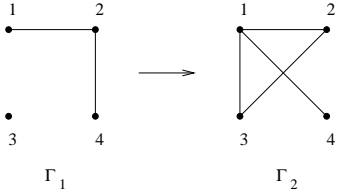}
\caption{$\Gamma_1$ and $\Gamma_2$ are switching-related with respect to
$\{1,2\}\subset \{1,2,3,4\}$}
\label{switching}
\end{figure}

Encoding adjacency, respectively non-adjacency, of distinct vertices
by $1$, respectively $-1$, yields a bijection between 
switching classes of graphs
and switching classes of matrices.
Conjugation by permutation matrices corresponds to relabeling the vertices of
a graph $\Gamma$. Conjugation by a diagonal $\pm 1-$matrix
corresponds
to the substitution of $\Gamma$ by a switching-related graph.
\end{rem}

\section{Switching classes and vorticity}\label{linking}

The {\it set of vorticities} (also called {\it homological equivalence
class or chiral signature}), introduced by Viro \cite{V},
is a classical and
well-known invariant for configurations of skew lines. We sketch below
briefly the well-known proof that it corresponds to the switching
class of an associated linking matrix.

We prefer to work with switching classes corresponding to
symmetric matrices (up to conjugation by signed permutation matrices)
with zero diagonal and off-diagonal coefficients in $\{\pm 1\}$.

\medskip

The {\it vorticity} $\hbox{vort}(L_i,L_j,L_k)$ of three lines
(see \cite[Section 2]{V} or \cite[Section 3]{CP}) is defined as
the product $x_{i,j}x_{j,k}x_{k,i}\in \{\pm 1\}$ of the signs
for the corresponding three crossings. The result
is independent of the chosen orientations for $L_i,L_j,L_k$,
classifies the skew-configuration $\{L_i,L_j,L_k\}$ up to
isotopy and
yields an invariant
$$\{\hbox{triplets of lines in configurations of skew lines}\}
\longrightarrow \{\pm 1\}\ .$$

Let us remark that almost all authors use the terminology
{\it linking coefficient} instead of vorticity.
This is slightly confusing since the linking coefficient
denotes also the isotopy type of
a pair of oriented skew lines.

The {\it set of vorticities} is the list of vorticities
$\hbox{vort}(L_i,L_j,L_k)$ for  all
triplets of lines $\{L_i,L_j,L_k\}$ in a configuration of skew lines.

Sets of vorticities (defining a {\it two-graph}, see \cite{Za1}) and
switching classes are equivalent. Indeed, vorticities of
a configuration of skew lines $\mathcal C$ can easily be retrieved from
a linking matrix for
$\mathcal C$. Conversely, given all vorticites ${\rm vort}(L_i,L_j,L_k)$
of a configuration $\mathcal C=\{L_1,\dots,L_n\}$, 
choose an orientation of the first line $L_1$.
Orient the remaining lines
$L_2,\dots,L_n$ such that they cross $L_1$ positively.
A linking matrix $X$ for $\mathcal C$ is given by $x_{1,i}=
x_{i,1}=1$ for $i$ such that $2\leq i\leq n$
and $x_{a,b}={\rm vort}(L_1,L_a,L_b)$ for $2\leq a\not= b\leq n$.

Two configurations of skew lines are {\it homologically equivalent}
(the terminology refers to properties of the complement, endowed with
a suitable extra-structure, of
a configuration in $\mathbb{RP}^3$) if there
exists a bijection
between their lines, which preserves all vorticities.
Two configurations are {\it homologically equivalent} if and only
if they have switching-equivalent linking matrices.

The sign indeterminacy in linking matrices representing switching classes
makes their use more difficult.
A satisfactory answer addressing this problem will be given
in Section \ref{oddswitching} for switching classes of odd order. 
For even orders, there seems to be no
completely satisfactory way to get rid of all sign-indeterminacies,
see Section \ref{even_case}.

\section{Euler partitions}\label{euler}

In this section we study invariants of computational cost $O(n^2)$
for switching classes of order $n$.

The behaviour of switching classes depends on
the parity of their order.

Switching classes of odd order $2n-1$ are in bijection with Eulerian
graphs. This endows the lines of a skew configuration
consisting of an odd number of 
lines with a {\it semi-orientation} (a canonical orientation,
up to global change) which we call the {\it Eulerian semi-orientation}.
An Eulerian semi-orientation induces a partition of the lines into equivalence
classes by counting their number of positive crossings.
We consider the case of odd order in Section \ref{oddswitching}.

The situation for switching classes of even order $2n$ is more
complicated. We replace Eulerian graphs appearing for odd orders
by a suitable kind of rooted binary trees which we call
{\it Euler trees}. The leaves of the Euler tree induce again a
natural partition, called the {\it Euler partition},
of the set of lines into equivalence classes
of even cardinalities. Section
\ref{even_case} deals with the even case.

\subsection{Switching classes of odd order - Eulerian semi-orientations}
\label{oddswitching}

A simple finite graph $\Gamma$ is {\it Eulerian} if all its
vertices are of even degree. 
Figure \ref{Euler_graphs_5} shows all seven Eulerian graphs on $5$ vertices.

\begin{figure}[h]
\epsfysize=4cm
\centerline{\epsfbox{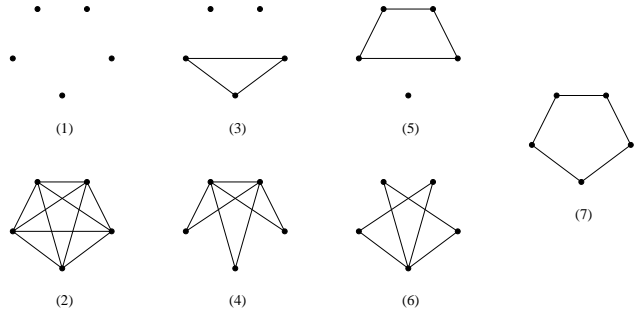}}
\caption{All Eulerian graphs on $5$ vertices}\label{Euler_graphs_5}
\end{figure}

The following well-known result goes back
to Seidel \cite{Se}.

\begin{prop} \label{Eulerbijection}
Eulerian graphs with an odd number $2n-1$ of vertices are
in bijection with switching classes of order $2n-1$.
\end{prop}

We recall a simple proof of Proposition \ref{Eulerbijection}
since it yields a fast algorithm for computing
Eulerian semi-orientations on configurations with an odd number 
of skew lines.

\begin{proof}
Choose a representing matrix $X$ of a switching class.
For $i$ such that $1\leq i\leq 2n-1$ define the number
$$v_i=\sharp\{j\ \vert\ x_{i,j}=1\}=\sum_{j=1,j\not= i}^{2n-1}
\frac{x_{i,j}+1}{2}$$
counting all entries equal to $1$ in the $i-$th row of $X$.
Since $X$ is symmetric, the vector $(v_1,\dots,v_{2n-1})$
has an even number of odd
coefficients and conjugation of the matrix $X$ with the diagonal matrix having
diagonal entries $(-1)^{v_i}$ turns $X$ into a matrix $X_E$ having an even number
of $1$'s in each row and column. The matrix $X_E$ is well-defined up
to conjugation
by a permutation matrix and defines an Eulerian graph $\Gamma_X$
with vertices
$\{1,\dots,2n-1\}$ and edges $\{i,j\}$ if $(X_E)_{i,j}=1$.
The Eulerian graph $\Gamma_X$ is unique up to relabeling its
vertices since switching with respect to a non-trivial subset 
of vertices destroys the Eulerian property of
$\Gamma_X$.
\end{proof}

A {\it semi-orientation} of a set of lines $\mathcal L$ is an orientation
of all lines in $\mathcal L$, up to global inversion of all orientiations.

Let $\mathcal C$ be a configuration of skew lines having an odd number of 
lines. Label and orient the lines of $\mathcal C$
arbitrarily in order to get a linking matrix $X$. Inverting the
orientations of all lines having an odd number of positive crossings
we get a unique semi-orientation which we
call the {\it Eulerian semi-orientation} of $\mathcal C$.

An {\it Eulerian linking matrix} $X_E$ associated to an Eulerian
semi-orientation of $\mathcal C$ is uniquely defined up 
to conjugation by a
permutation matrix. Its invariants
coincide with those of the switching class of $X_E$ but are slightly
easier to compute since there is no sign ambiguity. In particular,
some of them can be computed using only $O(n^2)$ operations.

An {\it Eulerian partition} of the set of lines of a configuration
consisting of an odd number of lines is by definition
the partition of the lines into subsets ${\mathcal L}_k$
consisting of all lines involved in exactly $2k$
positive crossings for an Eulerian semi-orientation.

\medskip

A few more invariants of Eulerian matrices are:

\begin{enumerate}

\item The total sum $\sum_{i,j} x_{i,j}$ of all entries in an Eulerian
linking matrix $X_E$ (this is of course equivalent to the computation of
the number of entries equal to $1$ in $X_E$). 
The computation of this invariant needs only $O(n^2)$ operations.

\item Its signature $\epsilon=\prod_{i<j} x_{i,j}$. The easy
identity
$$\epsilon=(-1)^{\left(-n(n-1)+\sum_{i,j}x_{i,j}\right)/4}$$
relates the signature to the total sum $\sum_{i,j} x_{i,j}$ of all entries
in an Eulerian linking matrix.

\item The number of rows of $X_E$ with given row-sum.
These numbers yield of course the
cardinalities of the sets ${\mathcal L}_0,{\mathcal L}_1,\dots$ and can 
be computed using $O(n^2)$ operations.

\item All invariants of the associated
Eulerian graph (having edges corresponding to entries
$x_{i,j}=1$) defined by $X_E$, e.g. the number of triangles or of
other fixed subgraphs. In particular, one can consider the number
$a_{i,j}$ of edges joining a vertex of degree $2i$ to a vertex of
degree $2j$.

\end{enumerate}

\medskip

For example, for $7$ vertices, there are $54$ different Eulerian
graphs, $36$ different sequences 
of vertex degrees (up to a permutation of the vertices), 
and $18$ different numbers for the cardinality of 1's in $X_E$.

\subsection{Switching classes of even order - Euler partitions}
\label{even_case}

The situation in this case is more complicated and less
satisfactory.

Given a matrix $X$ representing a
switching class with $2n$ vertices,
there exists a natural partition of the $2n$ rows $R$ of $X$ into two subsets
$R_+$ and $R_-$ according to the sign
$$\epsilon_i=\prod_{j\not=i}x_{i,j}\ \prod_{s<t}x_{s,t}$$
associated to the $i-$th row of $X$. This sign is well-defined
since switching (conjugation) with respect to a diagonal 
$\{\pm 1\}-$matrix $D$
multiplies both factors $\prod_{j\not=i}x_{i,j}$
and $\prod_{s<t}x_{s,t}$ by $\det (D)\in \{\pm 1\}$.

Since $\prod_i \epsilon_i=\prod_{i,j,\ i\not=j }x_{i,j}
\left(\prod_{s<t}x_{s,t} \right)^{2n}=1$,
both subsets $R_+,R_-$ have
even cardinalities.

If $R_+$ (or equivalently, $R_-$) is non-empty, it defines a symmetric
submatrix $X_+$ of even size $\sharp(R_+)$ corresponding to all
rows and columns with indices in $R_+$. Iterating the above construction
we get a partition $R_+=R_{++}\cup R_{+-}$. This construction is
most conveniently encoded by a rooted binary tree embedded in the
oriented plane which we
call the {\it Euler tree} of $X$: Draw a root $R$ corresponding
to the row-set $R$ of $X$. If the partition $R=R_+\cup R_-$ is
non-trivial, join the root $R$ to  a left successor called
$R_-$ and a right successor called $R_+$. The Euler tree of $X$ is now
constructed recursively by gluing the root $R_\pm $ of the Euler tree
associated to $X_\pm $ onto the corresponding successor $R_\pm $
of the root $R$.

The leaves of the Euler tree $T(X)$ of $X$ correspond to subsets
$R_w$ (with $w\in\{\pm\}^*$) of even cardinalities $2n_w$
summing up to $2n$. Leaves of $T(X)$ define
symmetric submatrices in $X$ which we call {\it Eulerian}:
All their row-sums are identical modulo $2$ and can be chosen to
be even, after a suitable conjugation.
The row partition $R=R_+\cup R_-$ of an Eulerian matrix $X$
of even size is by definition trivial. The sign $\epsilon\in\{\pm 1\}$
defined as $\epsilon=\epsilon_i$ for $i$ an arbitrary row of $X$
is called the {\it signature} of the Eulerian matrix $X$.
An Eulerian matrix of size $2$ has always signature $1$. For
Eulerian matrices of size $2n\geq 4$ both signs can occur
as signatures
since changing the signs of the entries $x_{i,j},\ 1\leq i\not= j\leq 3$,
inverts the signature of an Eulerian matrix (and preserves the set of
Eulerian matrices of even size $\geq 4$).
The signature of an Eulerian matrix encodes the parity
of the number of edges in an Eulerian graph (having only vertices
of even degrees) in the switching class of $X$.

The leaves of the Euler tree define a natural partition of the set of
rows of $X$ into subsets. We call
this partition the {\it Euler partition}.

An Euler tree is {\it signed} if its leaves are endowed with signs
$\pm 1$ corresponding to the signs of the associated
Eulerian matrices. An Euler tree is {\it weighted} if its leaves are endowed 
with strictly positive natural weights, a weight $m$ corresponding to 
an Eulerian matrix of size $2m\times 2m$. A {\it signed weighted}
Euler tree is both signed and weighted.

\begin{exa} \label{exampletree} The symmetric matrix
$$\left(\begin{array}{rrrrrrrrrr}
 0 & -1 &  1 & -1 &  1 &  1 &  1 &  1 &  1 & -1 \cr
-1 &  0 & -1 &  1 & -1 &  1 &  1 &  1 & -1 & -1 \cr
 1 & -1 &  0 & -1 &  1 &  1 &  1 & -1 & -1 &  1 \cr
-1 &  1 & -1 &  0 & -1 &  1 &  1 & -1 &  1 & -1 \cr
 1 & -1 &  1 & -1 &  0 &  1 & -1 & -1 &  1 &  1 \cr
 1 &  1 &  1 &  1 &  1 &  0 &  1 & -1 &  1 & -1 \cr
 1 &  1 &  1 &  1 & -1 &  1 &  0 & -1 & -1 &  1 \cr
 1 &  1 & -1 & -1 & -1 & -1 & -1 &  0 &  1 &  1 \cr
 1 & -1 & -1 &  1 &  1 &  1 & -1 &  1 &  0 &  1 \cr
-1 & -1 &  1 & -1 &  1 & -1 &  1 &  1 &  1 &  0
\end{array}\right)$$
yields the Euler partition
$$
R_{--}=\{3,5\} \ \ \cup 
\ \ R_{-+}=\{6,10\}\ \ 
\cup\ \ R_+=\{1,2,4,7,8,9\}
$$
(where the Eulerian submatrix associated to $R_+$ has signature $1$).
The associated signed Euler tree (with leaves of respective weights $1,1$ and 
$3$) is depicted in Figure \ref{tree}.

\begin{figure}[h]
\epsfysize=3cm
\centerline{\epsfbox{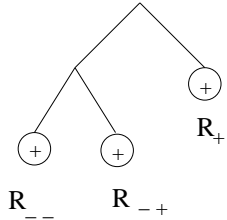}}
\caption{The signed Euler tree defined by Example
\ref{exampletree}}\label{tree}
\end{figure}

\end{exa}

\begin{exa}\label{exam_euler_even}
The characteristic polynomial of a
linking matrix of a configuration of skew lines is in general weaker than its
switching class: The linking matrices
$$\left(
\begin{array}{rrrrrrrr}
 0 &   1 &   1 &   1 &   1 &   1 &   1 &   1\\
 1 &   0 &  -1 &   1 &  -1 &   1 &  -1 &   1\\
 1 &  -1 &   0 &   1 &   1 &   1 &   1 &  -1\\
 1 &   1 &   1 &   0 &  -1 &   1 &   1 &  -1\\
 1 &  -1 &   1 &  -1 &   0 &   1 &   1 &  -1\\
 1 &   1 &   1 &   1 &   1 &   0 &  -1 &  -1\\
 1 &  -1 &   1 &   1 &   1 &  -1 &   0 &  -1\\
 1 &   1 &  -1 &  -1 &  -1 &  -1 &  -1 &   0
\end{array}\right)
$$
and
$$\left(
\begin{array}{rrrrrrrr}
 0 &   1 &   1 &   1 &   1 &   1 &   1 &   1\\
 1 &   0 &   1 &   1 &  -1 &   1 &  -1 &   1\\
 1 &   1 &   0 &   1 &  -1 &   1 &  -1 &   1\\
 1 &   1 &   1 &   0 &  -1 &   1 &   1 &  -1\\
 1 &  -1 &  -1 &  -1 &   0 &   1 &   1 &  -1\\
 1 &   1 &   1 &   1 &   1 &   0 &  -1 &  -1\\
 1 &  -1 &  -1 &   1 &   1 &  -1 &   0 &  -1\\
 1 &   1 &   1 &  -1 &  -1 &  -1 &  -1 &   0
\end{array}\right)
$$
are in different switching classes: The row-partition $R=R_-\cup R_+$ of the
first matrix is given by $R_-=\{2,5,7,8\}$ and $R_+=\{1,3,4,6\}$ 
with associated Eulerian
matrices $X_-$ and $X_+$ both of signature $1$. The second matrix
is Eulerian with signature $1$.
On the other hand, they have the same characteristic polynomial
$$(t-3)(t-1)^2(t+1)(t+3)^2(t^2-2t-11)\ .$$
This example is minimal in the sense that distinct switching
classes of order less than $8$ have distinct characteristic polynomials.
\end{exa}

Let us mention a last invariant related to the
Euler tree for a switching class $X$ having even order $2n$.
Let $R_1,\dots,R_m\subset R$ be the Euler partition of $X$ .
For $1\leq i\leq j\leq m$, define numbers $a_{i,j}
\in \{\pm 1\}$ by
$$a_{i,j}=\left\lbrace\begin{array}{ll}
\displaystyle
\prod_{t\not= s_{i_0}\in R_i} x_{s_{i_0},t}\prod_{s,t\in R_i,s<t} x_{s,t}
\quad&i=j\\
\displaystyle \prod_{s\in R_i,\ t \in R_j}x_{s,t}&i\not=j\end{array}
\right.
$$
where $s_{i_0}\in R_i$ is a fixed element.
Note that the number $a_{i,i}$ is the signature of the Eulerian matrix defined by the rows (and columns) of the set $R_i$.
One can easily check that the numbers $a_{i,j}$ are
well-defined.

\begin{rem} The equivalence relation induced on lines by the Euler
partition is fairly coarse. It is for instance generally much rougher
that the equivalence relation given by homologous lines
defined by Viro \cite {V}.
\end{rem}

\subsection{Enumerative aspects} It is natural to enumerate (signed) weighted
Euler trees according to the total sum $n$ of all weights.

The generating function $F(z)=
\sum_{n=0}^\infty \alpha_n z^n$ enumerating the number $\alpha_n$ of
distinct weighted Euler trees with total weight $n$
satisfies the equation
$$F(z)=\frac{1}{1-z}+\left(F(z)-1\right)^2$$
(with $\alpha_0=1$ corresponding to the empty tree).
Indeed, weighted Euler trees reduced to a leaf contribute $1/(1-z)$ to
$F(z)$. All other weighted Euler trees are obtained by gluing two 
weighted Euler
trees of strictly positive weights below a root and are enumerated
by the factor $\left(F(z)-1\right)^2$.

Solving for $F(z)$ we get the closed form
$$F(z)=\sum_{n=0}^\infty \alpha_nz^n=\frac{3(1-z)-\sqrt{(1-z)(1-5z)}}
{2(1-z)}.$$
showing that
$$\lim_{n\rightarrow\infty} \frac{\alpha_{n+1}}{\alpha_n}=5\ .$$
The first terms $\alpha_0,\alpha_1,\dots$ are given by
$$1,1,2,5,15,51,188,731,2950,12235,\dots,$$
see also Sequence A7317 of \cite{Sl}.

Similarly, the generating function $F_s(z)=
\sum_{n=0}^\infty \beta_n z^n$ enumerating the number $\beta_n$ of signed
weighted Euler
trees (keeping also track of the signature of all leaves with weight $\geq 2$)
with total weight $n$ satisfies the equation
$$F_s(z)=\frac{1+z^2}{1-z}+\left(F_s(z)-1\right)^2\ .$$
We get thus
$$F_s(z)=\sum_{n=0}^\infty \beta_nz^n=\frac{3(1-z)-\sqrt{(1-z)(1-5z-4z^2)}}
{2(1-z)}.$$
and
$$\lim_{n\rightarrow\infty} \frac{\beta_{n+1}}{\beta_n}=\frac{5+\sqrt{41}}{2}
\sim 5.7016\ .$$
The first terms $\beta_0,\beta_1,\dots$ are given by
$$1,1,3,8,27,104,436,1930,8871,41916,\dots, $$
see Sequence A110886 of \cite{Sl}.

\section{Spindle structures for switching classes}\label{spindles}

A construction of Mazurovski{\u\i} originally called the {\it isotopy
join} (or {\it spindle}) associates a configuration of $n$ skew 
lines to every permutation of $n$ letters.
We recall that a spindle is a configuration of skew lines with all lines
intersecting an oriented auxiliary line $A$, called its {\it axis}.
A {\it spindle-configuration}
(or a {\it spindle structure}) is a configuration of skew lines
isotopic to a spindle.

The orientation of the axis $A$ induces a linear order $L_1<\dots<L_n$
on the $n$ lines of a spindle $C$.
Each line $L_i\in C$ defines
a plane $\Pi_i$ containing $L_i$ and the axis $A$.

A second oriented auxiliary line $B$ (called a {\it directrix})
in general position with respect to $A,\Pi_1,\dots,\Pi_n$ and crossing
$A$ negatively,
intersects the planes $\Pi_1,\dots,\Pi_n$ at points
$\sigma(L_i)= B\cap \Pi_i$.
One can assume $\sigma(L_i)\in L_i$ by a suitable rotation
fixing $A\cap \Pi_i$ of the plane $\Pi_i$ containing $L_i$.
Since the orientation of $B$ induces a linear order on the points
$\sigma(L_i)$, we get a spindle-permutation (still denoted)
$i\longmapsto \sigma(i)$ of the set $\{1,\dots,n\}$ by identifying
the two linearly ordered sets $L_1,\dots,L_n$ and $\sigma(L_1),\dots,
\sigma(L_n)$ in the obvious way with $\{1,\dots,n\}$.
Figure \ref{spindle} displays an example corresponding  to
$\sigma(1)=1,\ \sigma(2)=4,\ \sigma(3)=2,\ \sigma(4)=5,\ \sigma(5)=3$.

\begin{figure}[h]
\epsfysize=5cm
\centerline{\epsfbox{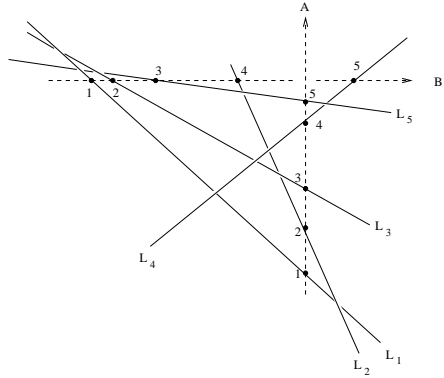}}
\caption{A spindle}\label{spindle}
\end{figure}

A linking matrix $X$ of a spindle $C$ is easily computed as follows.
Transform $C$ isotopically into a spindle with
oriented axis $A$ and directrix $B$ as
above. Orient a line $L_i$ from $L_i\cap A$ to $\sigma(L_i)=L_i\cap B$. A
straightforward
computation shows that the linking matrix $X$ of this labeled and
oriented configuration of skew lines has coefficients
$$x_{i,j}={\rm sign}((i-j)(\sigma(i)-\sigma(j)))$$
where $\hbox{sign}(0)=0$ and $\hbox{sign}(x)=\frac{x}{\vert x\vert}$
for $x\not=0$ and where
$\sigma$ is the corresponding spindle-permutation.

The linking matrix of Figure \ref{spindle} is 
$$X=\left(\begin{array}{rrrrr}
0 &  1 &  1 &  1 &  1\cr
1 &  0 & -1 &  1 & -1\cr
1 & -1 &  0 &  1 &  1\cr
1 &  1 &  1 &  0 & -1\cr
1 & -1 &  1 & -1 &  0\end{array}\right)\ .$$

\medskip

Two spindle-permutations are {\it equivalent}
(see \cite[Section 15]{CP}), and give rise to isotopic spindle
configurations,
if they are equivalent under the equivalence relation generated by
\begin{enumerate}
\item (Circular move)
$$\sigma\sim \mu\hbox{ if }\mu(i)=(s+\sigma((i+t)\pmod n)) \pmod n$$
for some integers $0\leq s,t<n$ (all integers are modulo $n$).

\item (Vertical reflection of a block or local reversal)
$\sigma\sim \mu$ if $\sigma([1,k])=[1,k]$ and
$$\mu(i)=\left\{\begin{array}{ll}
k+1-\sigma(k+1-i)\qquad &i\leq k\cr
\sigma(i)&i> k
\end{array}\right.$$
for some integer $k\leq n$ (see Figure \ref{vertical_ref}).

\begin{figure}[h]
\epsfysize=3cm
\centerline{\epsfbox{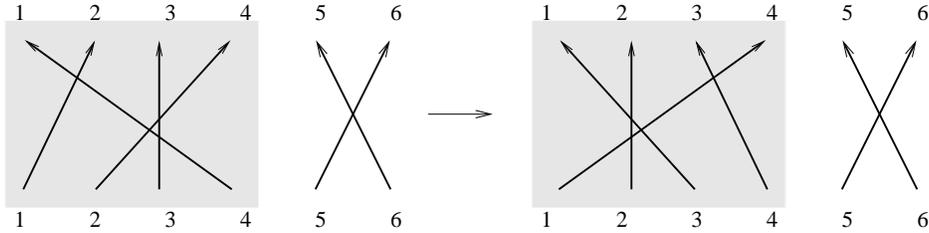}}
\caption{Vertical reflection of a block}\label{vertical_ref}
\end{figure}

\item  (Horizontal reflection of a block or local inversion)
$\sigma\sim \mu$ if there exists an integer $1<k\leq n$
such that $\sigma([1,k])=[1,k]$ and
$$\mu(i)=\left\{\begin{array}{ll}
\sigma^{-1}(i)\qquad &i\leq k\cr
\sigma(i)&i> k\end{array}\right.$$
(see Figure \ref{horizontal_ref}).

\begin{figure}[h]
\epsfysize=3cm
\centerline{\epsfbox{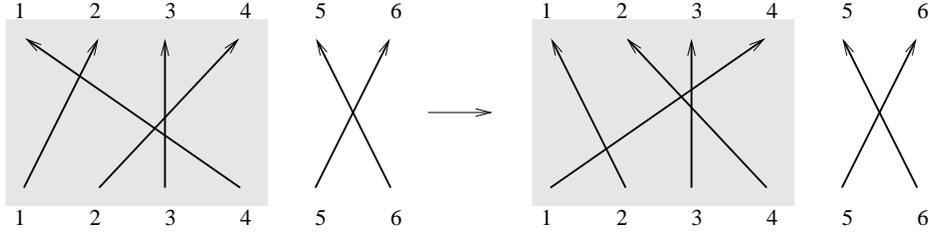}}
\caption{Horizontal reflection of a block}\label{horizontal_ref}
\end{figure}

\end{enumerate}

Permutations giving rise to
linking matrices in a common switching class are spindle-equivalent
by Corollary \ref{maincor}.

\section{An algorithm for a spindle-structure}\label{algo}

Given a switching class represented by a matrix $X$, the 
following algorithm constructs a spindle-permutation
(which is unique up to spindle-equivalence) with
linking matrix in the switching class of $X$
or proves non-existence of such a permutation.

\medskip
\begin{algo} \
\begin{itemize}
\item[{\bf Initial data.}] A natural number $n$ and a switching class
represented by a symmetric matrix
$X$ of order $n$ with rows and columns indexed by $\{1,\dots,n\}$
and coefficients $x_{i,j}$ satisfying

$$\begin{array}{ll}
x_{i,i}=0,\qquad &1\leq i\leq n\ ,\cr
x_{i,j}=x_{j,i}\in \{\pm 1\}\ ,\quad &1\leq i\not= j\leq n\ .
\end{array}$$

\item[{\bf Initialization.}] Conjugate the symmetric matrix $X$ by the
diagonal matrix with diagonal coefficients $(1,x_{1,2},x_{1,3},
\dots,x_{1,n})$. Set $\gamma(1)=\gamma(2)=1,\ \sigma(1)=1$ and $k=2$.

\item[{\bf Main loop.}]

Replace $\gamma(k)$ by $\gamma(k)+1$ and set
$$\sigma(k)=1+\sharp\{j\ \vert\ x_{\gamma(k),j}=-1\}+\sum_{s=1}^{k-1}
x_{\gamma(s),\gamma(k)}\ .$$

Check the following conditions:
\begin{enumerate}
\item $\gamma(k)\not=\gamma(s)$ for $s\in\{1,\dots,k-1\}$.
\item $x_{\gamma(k),\gamma(s)}=\hbox{sign}(\sigma(k)-\sigma(s))$ for
$s\in \{1,\dots,k-1\}$
(where $\hbox{sign}(0)=0$ and $\hbox{sign}(x)=\frac{x}{\vert x\vert}$
for $x\not=0$).
\item for $j\in \{1,\dots,n\}\setminus\{\gamma(1),\dots,\gamma(k)\}$ and
for $s\in\{1,\dots,k-1\}$:\\
if $x_{j,\gamma(s)}\ x_{\gamma(s),\gamma(k)}=-1$, then
$x_{j,\gamma(k)}=x_{j,\gamma(s)}$.
\end{enumerate}

\medskip

\noindent
If all conditions are fulfilled, then:
\begin{itemize}
\item[] if $k=n$, print all the data (mainly the spindle-permutation
$i\longmapsto \sigma(i)$ and perhaps also the conjugating permutation
$i\longmapsto \gamma(i)$) and stop.
\item[] if $k<n$, then set $\gamma(k+1)=1$, replace $k$ by $k+1$
and iterate the main loop.
\end{itemize}

\medskip

\noindent
If at least one of the above conditions is not fulfilled, then:
\begin{itemize}

\item[] while $\gamma(k)=n$ replace $k$ by $k-1$.

\item[] if $k=1$, print \lq\lq no spindle structure exists for
this switching class'' and stop.

\item[] if $k>1$, iterate the main loop.

\end{itemize}

\end{itemize}
\end{algo}

\subsection{Explanation of the algorithm}

The {\bf initialization} is simply a
normalization: we assume that the first row of the matrix represents the
first line of a spindle-permutation $\sigma$ normalized to $\sigma(1)=1$
(up to a circular move, this can always be done for a
spindle-permutation).

The {\bf main loop} assumes that row number $\gamma(k)$ of $X$ contains
the linking numbers of the $k-$th line $L_k$
(supposing a correct possible choice of the rows encoding the
linking numbers of $L_1,\dots,
L_{k-1}$).
The image $\sigma(k)$ of $k$ under a spindle-permutation is
then uniquely defined and given by the formula used in the main loop.

One has to check three necessary conditions:
\begin{itemize}
\item The first condition checks that row number $\gamma(k)$ has
not been used before.
\item The second condition checks the consistency of the choice for $\gamma(k)$
with all previous choices.
\item If the third condition is violated, the choice of rows
$\gamma(1),\dots,\gamma(k)$ leads to a dead end. Indeed, we
have then either
$x_{\gamma(s),\gamma(k)}=1,x_{j,\gamma(s)}=-1$ or
$x_{\gamma(s),\gamma(k)}=-1,x_{j,\gamma(s)}=1$ for some index
$j\not\in\{\gamma(1),\gamma(2),\dots,\gamma(k)\}$ and some natural
integer $s<k$. In the first case, the line-segments $L_s$ and
$L_k$ with $s<k$ representing $\sigma$ graphically do not cross.
This shows that any line-segment $L_m$ with $m>k$ which crosses $L_s$
has to cross $L_k$ first. Since there must be at least one such line
segment corresponding to the choise $\gamma(j)=m$, the algorithm
must backtrack. The second case is similar.

\end{itemize}

The algorithm runs correctly even without checking out Condition (3).
However, it loses much of its interest:
An instance of Condition (3) (with fixed $j,s,k$) is violated
with probability $\frac{1}{4}$ for a ``random'' choice
(made e.g. by flipping a fair coin) of
$x_{j,\gamma(s)},x_{j,\gamma(k)}\in \{\pm 1\}$.
This ensures fast running time in the average,
as observed experimentally.

The algorithm, if successful, produces two permutations
$\sigma$ and $\gamma$. The linking matrix of the spindle
permutation $\sigma$ is in the switching class of $X$ and
$\gamma$ yields a conjugation between these two matrices.
More precisely:
$$x_{\gamma(i),\gamma(j)}=\hbox{sign}((i-j)(\sigma(i)-\sigma(j)))$$
under the assumption $x_{1,i}=x_{i,1}=1$ for $2\leq i\leq n$.

Failure of the algorithm (indicated by the output
``no spindle structure exists for this switching
class'') proves non-existence of a spindle structure in
the switching class of $X$.

\section*{Acknowledgments}
The second author wishes to thank the Institut Fourier where most of
this work was done and Mikhail Zaidenberg for hosting his stay.

\end{document}